\documentclass[12pt]{amsart}
\usepackage{fullpage,url,amssymb,enumerate,colonequals}
\usepackage{mathtools}
\usepackage{mathrsfs} % for \mathscr (script letters)
\usepackage{caption}
\usepackage{MnSymbol}
\usepackage{extarrows}
\usepackage{lscape}
\usepackage[all,cmtip]{xy}
\usepackage{tablefootnote}
\usepackage[OT2,T1]{fontenc}
\usepackage{color}

\usepackage{xr-hyper}

\usepackage[
        colorlinks, citecolor=darkgreen,
        backref,
        pdfauthor={}, % add other authors
]{hyperref}
\usepackage{cleveref}
\usepackage{comment}
\usepackage{todonotes}

\usepackage[pdftex]{epsfig}

\numberwithin{equation}{section}

\newtheorem{theorem}[equation]{Theorem}

\newtheorem{claim*}{Claim}

\theoremstyle{definition}

\definecolor{darkgreen}{rgb}{0,0.5,0}
\definecolor{rem}{rgb}{0.8,0,0}
\definecolor{new}{rgb}{0.3,0.1,0.9}
\definecolor{reply}{rgb}{0,0,0.8}
\definecolor{gray}{gray}{0.7}

\renewcommand{\gcd}{\text{gcd}}
\newcommand{\F}{\mathbb{F}}

\newcommand{\Q}{\mathbb{Q}}
\newcommand{\R}{\mathbb{R}}
\newcommand{\Z}{\mathbb{Z}}
\newcommand{\rhobar}{{\overline{\rho}}}

\newcommand{\Xp}{X_E^{-}(p)}

\newcommand{\Frob}{\text{Frob}}

\newcommand{\oE}{\overline{E}}
\newcommand{\oF}{\overline{F}}
\newcommand{\GL}{\text{GL}}

\newcommand{\id}{\text{id}}

\DeclareMathOperator{\End}{End}

\begin{document}

\title {Revisiting the equation $x^2 + y^3 = z^p$}

\author{Nuno Freitas}
\address{Instituto de Ciencias Matemáticas (ICMAT),
          Nicolás Cabrera 13-15
         28049 Madrid, Spain}
\email{nuno.freitas@icmat.es}

\author{Diana Mocanu}
\address{Max Planck Institute for Mathematics,
          Vivatsgasse 7,
         53111 Bonn, Germany
         }

         \email{d.mocanu@mpim-bonn.mpg.de}
\author{Ignasi S\'anchez-Rodr\'iguez}
\address{Universitat de Barcelona (UB),
        Gran Via de les Corts Catalanes 585,
        08007 Barcelona, Spain}
\email{ignasi.sanchez@ub.edu}

\thanks{Freitas was partly supported by the PID 2022-136944NB-I00 grant of the MICINN
(Spain)}
\thanks{S\'anchez-Rodr\'iguez was partly supported by the PID 2019-107297GB-I00 grant of the MICINN (Spain)}

\subjclass[2010]{}

\begin{abstract}
Let $E/\mathbb Q$ be an elliptic curve and $p \geq 3$ a prime. The modular curve~$X_E^-(p)$ parameterizes elliptic curves with $p$-torsion modules anti-symplectically isomorphic to~$E[p]$.

The work of Freitas–Naskręcki–Stoll uses the modular method to show that all primitive non-trivial solutions of the Fermat-type equation $x^2 + y^3 = z^p$
give rise to rational points on~$X_E^-(p)$ with $E \in 
\{27a1,54a1,96a1,288a1,864a1,864b1,864c1 \}$.
%belonging to a finite list of seven elliptic curves including $864a1$, $864b1$ and~$864c1$. 
Using a criterion classifying the existence of local points due to the first two authors, we show that, for $E$ any of the curves with conductor 864 and
certain primes~$p \equiv 19 \pmod{24}$, 
we have $X_E^-(p)(\Q_\ell) = \emptyset$.
Furthermore, for each $E$ in the list and any $p$, we prove that either $X_E^-(p)$ can be discarded using the same criterion, or it cannot be discarded using purely local information.
\end{abstract}
\maketitle

\section{Introduction}

Let $K$ be a field, $E/K$ an elliptic curve and $p \geq 3$ a prime. The non-cuspidal $K$-points of the curve~$X^-_E(p)$ parametrize pairs $(E', \phi)$ where $E'/K$ is an elliptic curve and $\phi : E[p] \to E'[p]$ a strictly anti-symplectic isomorphism of $G_K$-modules; the curve $X^+_E(p)$ parameterizes pairs where $\phi$ is a symplectic isomorphism instead; see \cite[\S 3]{FMlocal} for a reminder on~$X^\pm_E(p)$. 

For an elliptic curve $E/K$, we let $\rhobar_{E,p}$ denote its $p$-torsion representation.

For $E/\Q_\ell$ with good reduction denote by $\oE$ its reduction over $\F_\ell$. 
 Let $\pi\in \End(\oE)$ be the Frobenius endomorphism satisfying the equation $\pi^2-a_\ell \:\pi+\ell=0$ of discriminant 
$\Delta_\ell := a_\ell^2 - 4\ell$ where $a_\ell:=(\ell+1) - \# \oE(\F_\ell)$. Set also $b_E:=[\End (\oE): \Z[{\pi}]]$, and
%set $b_E:=[\End (E/\F_\ell): \Z[\pi]]$ and $\Delta_\ell := a_\ell(E)^2 - 4\ell$, where~$\pi:=\sqrt{\Delta_\ell}$ and $a_\ell(E)$ is the trace of Frobenius at $\ell$. 
note that $b_E$ is finite when $\oE$ is ordinary.

Let $K=\Q$.
In~\cite{FMlocal}, the first two authors show that $X^-_E(p)(\R)\neq \emptyset$ and 
give a complete classification of when $X^-_E(p)(\Q_\ell)\neq \emptyset$, for $\ell \neq p$. In particular, the case where $\ell$ is a prime of good reduction of~$E$ is covered by Theorem~5.3 and Remark 5.4 in {\it loc. cit.} 
and is as follows.
\begin{theorem}\label{thm:main}
    Let $\ell$ and $p\geq 3$ be different primes. Let $E/\Q$ have good reduction at~$\ell$. Then $X_E^-(p)(\Q_\ell)=\emptyset$ if and only if $E$ has ordinary reduction and all of the following hold:
    \begin{enumerate}
        \item $p\equiv 3 \pmod{4}$;
        \item $-p\Delta_\ell = s^2$ for some $s\in \Z$.
        \item $p\mid \#\rhobar_{E,p}(\Frob_\ell)$ or, equivalently, $p\mid \Delta_\ell$ and $p\nmid b_E$.
        \item for all primes $q\neq \ell$, $q\mid \Delta_\ell \implies (q/p)\neq -1$.
        \item $\ell<p^2/16$.
    \end{enumerate}
    where $\Frob_\ell$ is a Frobenius element at~$\ell$.
\end{theorem}
As an application of this result, in Section~\ref{sec:23p}, we revisit the Fermat-type equation~$x^2 + y^3 = z^p$ where $p \geq 7$ is a prime.
This equation was solved for $p=7$ by Poonen--Schaefer--Stoll~\cite{237} and for $p=11$ by Freitas--Naskr{\k e}cki--Stoll~\cite{23n} under GRH. The approach in {\it loc. cit.} involves two steps: first, use the modular method to show that a primitive solution to the equation gives a $\Q$-point on~$X_E^\pm(p)$ for $E$ one of the curves in~\eqref{list}; second, determine all the rational points of the relevant $X_E^\pm(p)$ and reconstruct the solutions (the expectation is that only the known solutions will emerge).

Applying the same strategy for~$p \geq 13$ is challenging because the genus of~$X_E^\pm(p)$ grows quickly with~$p$, causing existing techniques to be ineffective in finding rational points.
Note that the curve $X_E^+(p)(\Q)$ always contains the canonical point $(E, \id_{E[p]})$, while $\Xp(\Q)$ may be empty.
 Using Theorem~\ref{thm:main} we have an efficient way to show that $X_E^-(p)(\Q)=\emptyset$ by proving that it lacks $\Q_\ell$-points, for some carefully chosen $\ell \ne p$ where $E$ has good reduction. This complements the local analysis done in~\cite{23n} using primes of bad reduction, and largely improves that in~\cite{symplectic} using primes of good reduction.

 To prove the next theorem, we test the hypothesis of Theorem~\ref{thm:main} using {\tt Magma} for pairs~$(E,p)$ with $E$ in~\eqref{list} and $p < 1000$. We remark that the same code produces additional examples if run for larger values of $p$.

\begin{theorem}
    
\label{cor:main}
Let $(a,b,c)\in \Z^3$ satisfy $a^2 + b^3 = c^p$ and $\gcd(a,b,c) = 1$. Then $(a,b,c)$ does not give rise to a rational point on
$X_{E}^-(p)$ for the following pairs $(E,p)$:
\begin{table}[!ht] 
    \centering
    \begin{tabular}{|c||c|}
     \hline
        $E $& $p$ \\ \hline
         $864a1$& $19,43,211, 307, 499, 523, 547, 571, 739, 787, 859, 907$  \\ \hline
         $864b1$& $ 19, 43, 67, 307, 331, 571, 619, 643, 691, 787, 811, 859$\\ \hline
         $864c1$&  $211, 331, 499, 523, 571, 643, 691, 859$ \\ \hline
    \end{tabular}
 
\end{table}
\end{theorem}

Finally, in Section~\ref{sec:last}, we show that, for all $E$ in~\eqref{list} and any~$p\geq 7$  either
$\Xp(\Q_\ell) = \emptyset$ for some~$\ell \neq p$ (and hence Theorem~\ref{thm:main} is satisfied)
or $\Xp(\Q_\ell)$ contains points arising from $\ell$-primitive solutions of $x^2 + y^3 = z^p$ for all~$\ell \neq p$ (see~Theorem~\ref{thm:pts23p}), and so~$\Xp$ cannot be excluded by local information away from $p$.

\subsection{Acknowledgements}
We thank Alain Kraus for helpful discussions and remarks.
This work was completed while all the authors were at the Max Planck Institute for Mathematics in Bonn. We are grateful for its hospitality and financial support.

\section{The Fermat equation \texorpdfstring{$x^2 + y^3 = z^p$}{x² + y³ = zᵖ}}\label{sec:23p}

 %Consider the generalized Fermat equation
% \begin{equation} \label{E:GFE1}
%   x^r + y^q =Cz^p,
% \end{equation}
% where $p,q,r > 2$ are primes and $C\in \Z_{>0}$.
% A solution $(a,b,c) \in \Z^3$ of~\eqref{E:GFE1} is called {\it primitive} if $\gcd(a,b,c)=1$ and it is called {\it non-trivial} if $abc \neq 0$. 
% A triple $(a,b,c) \in \Z_\ell^3$ is called an {\it $\ell$-primitive solution} if it satisfies~\eqref{eq:23p} and  $\min\{v_\ell(a),v_\ell(b),v_\ell(c) \} = 0$. The triple of exponents $(p,q,r)$ are called the {\it signature} of the equation.

% The folklore expectation is that when $p,q,r$ are large enough, \eqref{E:GFE1} has no primitive non-trivial solutions, and this is well known to be a consequence of the \textit{abc} conjecture.
% The modular method is very powerful to prove the non-existence of solutions for large varying~$p$ but often fails for small values of~$p$. As an application of our main theorem, we will introduce a new technique for the elimination step of the modular method for specific small~$p$.

Consider the generalized Fermat equation
\begin{equation} \label{eq:23p}
  x^2 + y^3 = z^p,
\end{equation}
where $p > 5$ is a prime.
A solution $(a,b,c) \in \Z^3$ of~\eqref{eq:23p} is called {\it primitive} if $\gcd(a,b,c)=1$ and it is called {\it non-trivial} if $abc \neq 0$. 
A triple $(a,b,c) \in \Z_\ell^3$ is called an {\it $\ell$-primitive solution} if it satisfies~\eqref{eq:23p} and  $\min\{v_\ell(a),v_\ell(b),v_\ell(c) \} = 0$.

Note that, besides the trivial solutions, equation~\eqref{eq:23p} also admits the `Catalan' solutions $(\pm 3)^2 + (-2)^3 = 1^p$.
As mentioned in the introduction, the equation~\eqref{eq:23p} has been solved for $p=7$ and for $p=11$ in \cite{23n} under GRH. We now summarize the strategy in {\it loc. cit.}.

Let $p \geq 11$ be a prime and suppose that $(a,b,c) \in \Z^3$ is a primitive solution to~\eqref{eq:23p} with exponent~$p$. We consider the associated Frey elliptic curve
\begin{equation} \label{Eq:Frey}
    F_{a,b} \; \colon \; y^2 = x^3 + 3bx - 2a
\end{equation}
whose arithmetic invariants are
\begin{equation} \label{E:jInvariant}
  c_4 = -12^2 b, \qquad c_6 = -12^3 a, \qquad
  \Delta = -12^3 (a^2 + b^3) = -12^3 c^p.
\end{equation}
Consider also the seven elliptic curves (specified by their Cremona label):
 \begin{equation}\label{list}
     27a1,\; 54a1,\; 96a1,\; 288a1,\; 864a1,\; 864b1,\; 864c1.
 \end{equation}

In~\cite{23n}, the authors employ the modular method to
show that, for some $d \in \{ \pm 1, \pm 2, \pm 3, \pm 6 \}$, the
quadratic twist $dF_{a,b}$ of the Frey curve gives rise to a rational point
(satisfying certain $2$-adic and $3$-adic conditions)
on the modular curve $X_E^+(p)$ or $X_E^-(p)$,
where $E$ is one of the curves in~\eqref{list}. Therefore, it suffices to find the $\Q$-points of $X_E^\pm(p)$ for all $E$ in \eqref{list} to solve the equation \eqref{eq:23p}. We say that we discard $X_E^+(p)$ or $X_E^-(p)$ from the list whenever we can show that it has no relevant rational points.

The curves $X_E^+(p)(\Q)$ always contain the canonical point $(E,\id_{E[p]})$, but
some of the~$X_E^-(p)$ might
lack rational points.
In \textit{loc. cit.}, the authors use inertial types, local symplectic criteria at the bad reduction primes $\ell=2,3$, and global arguments to discard many curves $X_E^-(p)$; the remaining ones are given in Table~\ref{tab:pmod24}.

% Requires: \usepackage{multirow}
\begin{table}[h]
    \centering
    \begin{tabular}{|c||c|c|c|c|c|c|}
    \hline
    \( p \pmod{24} \) & \( 54a1 \) & \( 96a1 \) & \( 288a1 \) & \( 864a1 \) & \( 864b1 \) & \( 864c1 \) \\ \hline
    \( 5 \) &  \( - \) &  & & \( - \) & \( - \) & \(  - \) \\ \hline
    \( 7 \) & \( - \) &  &  &  &  &  \\ \hline
    \( 11 \) & & & \(-\) &  &  & \\ \hline
    \( 13 \) &  & \( -\) & &  &  &  \\ \hline
    \( 19 \) &  & \(-\) & \( -\) & \( - \) & \( - \) & \( - \) \\ \hline
\end{tabular}
    \caption{Curves $X_E^-(p)$ left to eliminate for $p\geq 11$; see~\cite[Table 4]{23n}.}
    \label{tab:pmod24}
\end{table}
In the work of Freitas--Kraus~\cite{symplectic},
the authors excluded from Table~\ref{tab:pmod24} the possibilities
$X^-_{864a1}(p)$ for $p=19,43$ and $X^-_{864b1}(p)$ for $p=19,43,67$ (see Theorem 23 in {\it loc. cit}).
They achieved this by showing that $X_E^-(p)(\Q_\ell)=\emptyset$ for some small primes $\ell$ of good reduction, namely,
\begin{equation}\label{NuKrtriples}
 (E, \ell, p) =
(864a1,  5,  19),   (864a1,  31,  43), (864b1,  7,  19),  (864b1,  13,  43),  (864b1,  19,  67).
\end{equation}
We will say that $(E,\ell,p)$ is an {\it exceptional triple} if $X_E^-(p)(\Q_\ell)=\emptyset$ where $E$ and~$p$ are given in Table~\ref{tab:pmod24} and $\ell \neq 2,3$; in particular, $\ell$ is a prime of good reduction for~$E$.

The triples in~\eqref{NuKrtriples} were found by first performing a brute-force search for $5 \leq \ell<200$ and $p<400$ such that $p \mid \#\rhobar_{E,p}(\Frob_\ell)$ and, secondly, testing whether
\cite[Theorem~12]{symplectic} applies with $E$ and $F_{a,b}$ for all $1 \leq a,b \leq \ell$ satisfying $F_{a,b}[p] \simeq E[p]$.
By applying Theorem~\ref{thm:main}, we are able to replace these heavy computations with a quick verification of hypotheses (1)--(5).

We note that Theorem~\ref{thm:main} implies that $X_E^-(p)(\Q_\ell) \neq \emptyset$ whenever $p \equiv 1 \pmod 4$, so there are no exceptional triples with $p \equiv 5, 13 \pmod{24}$.
We cannot find exceptional triples when $E$ has CM, because $p \nmid \#\rhobar_{E,p}(\Frob_\ell)$ in that case, as the image of $\rhobar_{E,p}$ is the normalier of a Cartan subgroup of $\GL_2(\F_p)$; thus, there are no exceptional triples with $E=288a1$.
Moreover,  $E=96a1$ admits a~$2$-isogeny and $E=54a1$ a $3$-isogeny, hence \cite[Corollary~3.6]{FMlocal} gives
$X_E^-(p)(\Q) \neq \emptyset$ for
$$(E, p \pmod{24})=(96a1,19),(54a1,7).$$
 In summary, exceptional triples can occur only for
\begin{equation}\label{pW}
    p\equiv 19 \pmod{24} \qquad \text{ and } \qquad E = 864a1, \; 864b1 \; \text{ or } \; 864c1.
\end{equation}
\subsection{Proof of Theorem~\ref{cor:main}}
We have run a calculation in {\tt Magma} checking for exceptional triples $(E,\ell,p)$ with $(E,p)$ in \eqref{pW} and $p<1000$ satisfying the conditions of Theorem~\ref{thm:main}. The code is available at \cite{code} and uses the work of Centeleghe~\cite{Centeleghe} and its associated {\tt Magma} package to check the condition $p \mid \#\rhobar_{E,p}(\Frob_\ell)$. %The pairs $(E,p)$ which admit an exceptional triple are given in Table~\ref{table}.

\section{Points arising from \texorpdfstring{$\ell$}{l}-primitive solutions}\label{sec:last}
To finish our analysis of the local information associated with equation~\eqref{eq:23p}, we will show that Table~\ref{tab:pmod24} is locally complete in the following sense.
Consider the following set
\[
\mathcal{X}_{E,\ell, p}:=\{P\in X_E^-(p)(\Q_\ell) | \: P:=(dF_{a,b}, \phi), \: (a,b,c)\in \Z_\ell \text{ is an }\ell\text{-primitive solution to }\eqref{eq:23p} \}
\]
where $dF_{a,b}$ is a quadratic twist of the Frey curve $F_{a,b}$ given in \eqref{Eq:Frey}.
Clearly, any primitive solution $(a,b,c) \in \Z^3$ to~\eqref{eq:23p} is
an $\ell$-primitive solution for all primes~$\ell$. Therefore, 
showing that $\mathcal{X}_{E,\ell, p} \subseteq X_E^-(p)(\Q_\ell)$ is empty, suffices to eliminate $(E,p)$ from Table~\ref{tab:pmod24}.
In this direction, the results in~\cite{23n} show that, for the bad reduction primes $\ell =2, 3$, the pairs $(E,p)$ in Table~\ref{tab:pmod24} have $\mathcal{X}_{E,2, p}\neq \emptyset$ or $\mathcal{X}_{E,3, p}\neq \emptyset$. In this section, we prove the analogous result for $\ell>3$, the primes of good reduction for $E$.
\begin{theorem}\label{thm:pts23p}  
 Let $(E,p)$ be as in Table~\ref{tab:pmod24} and $\ell>3$ a prime $\neq p$. Assume $X_E^-(p)(\Q_\ell) \neq \emptyset$. Then $\mathcal{X}_{E,\ell, p}\neq \emptyset$.
\end{theorem}
\begin{proof}
Since $X_E^-(p)(\Q_\ell) \neq 0$, there is an elliptic curve $E_0 / \Q_\ell$ and $\phi : E[p] \simeq E_0[p]$ an anti-symplectic isomorphism. Since $E$ has good reduction at~$\ell>3$, it follows that
$E_0[p]$ is unramified, so that $E_0$ has good reduction or (by the theory of the Tate curve) it has multiplicative reduction and $p \mid v(j(E_0))$.
In both cases, we will show that there is an $\ell$-primitive solution $a,b,c \in \Z_\ell$ to~\eqref{eq:23p} and $d\in \Z_\ell^*$, yielding a Frey elliptic curve $F_{a,b}/\Q_\ell$ given by  \eqref{Eq:Frey} such that
\begin{equation}\label{eq:W0Frey}
 dF_{a,b}[p] \stackrel{\varphi}{\simeq} E_0[p] \stackrel{\phi^{-1}}{\simeq} E[p]
\end{equation} where $\varphi$ is symplectic. As $\phi^{-1}$ is anti-symplectic by assumption, one gets that $(dF_{a,b}, \phi^{-1} \cdot \varphi)$ gives a point on $\mathcal{X}_{E,\ell, p}$, as desired.

Suppose first that $E_0 / \Q_\ell$ has good reduction. We will construct a Frey elliptic curve $F_{a,b}/\Q_\ell$ with good reduction and residual curve $\F_\ell$-isomorphic to $\overline{E_0}/\F_{\ell}$, the mod~$\ell$ reduction of~$E_0$. Then, \cite[Theorem 12]{symplectic} assures that \eqref{eq:W0Frey} holds with $d=1$.
The curve $\overline{E_0} / \F_\ell$ admits a short Weierstrass model $y^2 = x^3 + Ax + B$. Since $\ell \neq 2,3$ we can define $\oF_{a,b} / \F_\ell$ by
\[
\oF_{a,b}  : y^2 = x^3 + 3\bar{b} x - 2 \bar{a}, \qquad \bar{a} =\frac{-B}{2}u^6  , \qquad \bar{b} =   \frac{A}{3}u^4\qquad \text{for some }u\in \F_\ell^*
\]
which is an elliptic curve because
$$-12^3(\bar{a}^2+\bar{b}^3)=\Delta(\oF_{{a},{b}}) = \Delta(\overline{E_0})u^{12} \neq 0,$$
and it is clearly ismorphic to $\oE_0$. 
We claim that there exists a choice of $u\in \F_\ell^*$ such that $\bar{a}^2+\bar{b}^3$ is a $p$-th power. We define $$\alpha:=\frac{B^2}{4}-\frac{A^3}{27}\in \F_\ell^*$$ and note that the claim is equivalent to the existence of $u\in \F_\ell^*$ such that $\alpha \cdot u^{12}$ is a $p$-th power. If $p\nmid \ell-1$, then all elements of $\F_\ell^*$ are $p$-th powers, and we can take $u=1$. If~$p\mid \ell-1$ then $\F_\ell^*/(\F_\ell^*)^p\cong \Z/p\Z$. Moreover, the map $u \to \alpha \cdot u^{12} \pmod{\F_\ell^{*p}}$ is surjective onto~$\F_\ell^*/(\F_\ell^*)^p$,
because $\gcd(12,p)=1$ (as $p>3$). Thus there is $u \in \F_\ell^*$ making $\bar{a}^2+\bar{b}^3=\alpha \cdot u^{12}$ a $p$-th power.

Finally, we denote the $p$-th root of  $\bar{a}^2+\bar{b}^3$ by $\bar{c}\in \F_\ell^*$.
Any lifts $a,b,c \in \Z_\ell$ of $\bar{a}$, $\bar{b}, \bar{c}$ yield a Frey curve $F_{a,b}$ with the required properties. Indeed, since $\overline{E_0}$ is non-singular
it follows that $\bar{a} \neq 0$ or $\bar{b} \neq 0$ and hence $a,b,c$ is $\ell$-primitive.

Suppose now that $E_0/\Q_\ell$ has multiplicative reduction with $v(j({E_0}))=-kp$, for $k\in \Z_{>0}$. In this case, we will construct a Frey elliptic curve $F_{a,b}/\Q_\ell$ with $j(F_{a,b})=j({E_0})\neq 0,1728$. This implies that $F_{a,b}/\Q_\ell$ or a quadratic twist of it is isomorphic to $E_0/\Q_\ell$
%gives rise to `the canonical point' on $X^+_{E_0}(p)(\Q_\ell)$ 
which immediately gives \eqref{eq:W0Frey}.

We will show that there is a choice of $\ell$-primitive $a,b,c\in \Z_\ell$ such that
\begin{equation} \label{eq:jinv}
\begin{cases}
    a^2+b^3=c^p,\\
    j(F_{a,b})=12^3\frac{b^3}{c^p}=j({E_0}).
\end{cases}
\end{equation}
Since $j({E_0})=-kp$ we need to pick $a,b,c$ with $v_\ell(a)=v_\ell(b)=0$ and $v_\ell(c)=k$. By writing $j(E_0)=u_{E_0}\cdot \ell^{-kp}$ and $c:=u_c\cdot \ell^{k}$, for some $u_{E_0},u_c \in \Z_\ell^*$ one gets that \eqref{eq:jinv} is equivalent to
\begin{equation} 
\begin{cases}
    12^3b^3=u_{E_0}\cdot u_c^p,\\
    a^2=c^p-\frac{u_{E_0}\cdot u_c^p}{12^3}.
\end{cases}
\end{equation}
We want to show that such $a,b,c\in \Z_\ell$ exist. By reducing everything modulo $\ell$ and Hensel's lemma, this is equivalent to showing that one can find $u_c\in \F_\ell^*$ such that
\begin{equation} \label{eq:uc}
\begin{cases}
    u_{E_0}\cdot u_c^p,\: \text{ is a cube in }\F^*_\ell,\\
    (-u_{E_0}\cdot u_c^p)/3 \text{ is a square in }\F^*_\ell.
\end{cases}
\end{equation}
If $3\nmid \ell-1$, then everything is a cube in $\F_\ell^*$. Thus, any choice of $u_c$ making $(-u_{E_0}\cdot u_c)/3$ a square works. If $3\mid \ell-1$, then $-1/3$ is a square in $\F_\ell^*$, so \eqref{eq:uc} is equivalent to $u_{E_0}\cdot u_c^p$ being a sixth power in $\F_\ell^*$. As in this case $6\mid \ell-1$, then $\F_\ell^*/(\F_\ell^*)^6\cong \Z/6\Z$. Since $\gcd(6,p)=1$ for $p>3$, one gets that $u_c \to u_{E_0}\cdot u_c^p \pmod{\F_\ell^{*6}}$ is surjective onto $\F_\ell^*/(\F_\ell^*)^6$, and thus there is a choice of $u_c$ making $u_{E_0}\cdot u_c^p$ a $6$-th power, concluding the proof.
\end{proof}

\end{document}